\documentclass[11pt,a4paper]{amsart}
\usepackage{latexsym}
\usepackage{amssymb}
\usepackage[all]{xy}

\newtheorem{theorem}{Theorem}[section]
\newtheorem{corollary}[theorem]{Corollary}
\newtheorem{lemma}[theorem]{Lemma}
\newtheorem{proposition}[theorem]{Proposition}
\newtheorem{definition}[theorem]{Definition}
\newtheorem{example}[theorem]{Example}
\newtheorem{remark}[theorem]{Remark}
\newcommand{\Hom}{{\rm Hom}}

\newcommand{\Ker}{{\rm Ker}}
\newcommand{\Coker}{{\rm Coker}}

\newcommand{\fp}{{\rm fp}}
\newcommand{\op}{{\rm op}}
\newcommand{\Ab}{{\rm Ab}}
\newcommand{\End}{{\rm End}}

\newcommand{\A}{\mathcal A}
\newcommand{\F}{\mathcal F}

\newcommand{\C}{\mathcal C}
\newcommand{\E}{\mathcal E}
\newcommand{\D}{\mathcal D}
\newcommand{\I}{\mathcal I}
\newcommand{\Z}{\mathbb Z}
\newcommand{\Q}{\mathbb Q}
\begin{document}
\sloppy

\title{One-sided exact categories}

\author{Silvana Bazzoni}

\address{Dipartimento di Matematica Pura e Applicata, Universit\`{a} di Padova, Via Trieste 63, 35121 Padova,
Italy} \email{bazzoni@math.unipd.it}

\author{Septimiu Crivei}

\address{Faculty of Mathematics and Computer Science \\ ``Babe\c s-Bolyai" University \\ Str. Mihail Kog\u alniceanu 1
\\ 400084 Cluj-Napoca, Romania} \email{crivei@math.ubbcluj.ro}


\begin{abstract} One-sided exact categories appear naturally as instances of Grothendieck pretopologies. In an
additive setting they are given by considering the one-sided part of Keller's axioms defining Quillen's exact
categories. We study one-sided exact additive categories and a stronger version defined by adding the
one-sided part of Quillen's ``obscure axiom''. We show that some homological results, such as the Short Five Lemma
and the $3\times 3$ Lemma, can be proved in our context. We also note that the derived category of a one-sided exact
additive category can be constructed.
\end{abstract}

\thanks{First named author supported by MIUR, PRIN 2007, project ``Rings, algebras, modules and categories'' and by
Universit\`{a} di Padova (Progetto di Ateneo CPDA105885/10 ``Differential Graded Categoires''). Second named author acknowledges the support of the
Romanian grant PN-II-ID-PCE-2008-2 project ID\_2271. He would like to thank the members of the Department of
Mathematics, and especially the first author, for the kind hospitality during his visits at Universit\`{a} di Padova.}

\date{May 26, 2011}

\maketitle

\section{Introduction}

The framework of exact categories has naturally appeared in order to develop homological algebra in
a categorical setting more general than that of an abelian category. Several notions
of exact categories have been defined in the literature, for instance, by Barr \cite{Barr}, Heller \cite{Heller},
Quillen \cite{Q} (simplified by Keller \cite{Keller}), or Yoneda \cite{Y}, to mention only some of the most
representative ones. Their importance was underlined by the broad range of applications in algebraic geometry, algebraic
and functional analysis, algebraic $K$-theory etc. Several recent papers give rather exhaustive accounts on exact
categories in the sense of Quillen-Keller, defined by means of a distinguished class of kernel-cokernel pairs, called
conflations, e.g. \cite{Buhler} or \cite{FS}. Also, these exact categories have been recently shown to provide a
suitable setting for developping an approximation theory \cite{MS}. 

During the last decade the natural occurrence of one-sided exact categories has become apparent in connection with 
the notion of Grothendieck pretopology in the sense of \cite[Expos\'e II, Definition 1.3]{SGA4}. Such a
pretopology in an arbitrary category $\C$ is given by assigning to each object $U$ of $\C$ a family of arrows
ending in $U$, called coverings of $U$, satisfying certain axioms. Rosenberg \cite{Rosenberg} recently
introduced arbitrary right exact categories, whose axioms mean that a distinguished
class of strict epimorphisms yields the coverings of a Grothendieck pretopology. He showed that right exact
categories offer a suitable framework for homological theories which could appear in non-commutative algebraic
geometry, whereas left exact categories (that is, categories whose opposite categories are right exact) give the
possibility to develop a more general version of $K$-theory. In an additive setting, Rump briefly considered left exact
categories \cite[Definition~4]{Rump2} defined by means of a distinguished class of cokernels in connection with the
problem of proving the existence of flat covers in non-abelian categories. In particular, one-sided almost
abelian (also termed quasi-abelian) categories in the sense of Rump \cite{Rump1} have a natural structure of one-sided
exact category given by the class of all kernel-cokernel pairs. 

We shall consider here the context of an additive category, and we shall define and study one-sided exact categories.
Left exact categories in our sense are given by means of a distinguished class of cokernels, called
deflations, satisfying certain axioms. Considering the deflations ending in an object $U$ of a left exact category as
the coverings of $U$, the axioms of a left exact category are nothing else than those of a Grothendieck pretopology.
Unlike Rosenberg, we shall be mainly interested in the context of additive categories, where more results can be
obtained. We shall also study strongly one-sided exact categories, which are defined by adding half of Quillen's
``obscure axiom'' to the axioms of a one-sided exact category. It turns out that many results on exact categories do not
use this extra axiom, and we would like to emphasize this fact by separating the two notions.
Note that our strongly one-sided exact categories are nothing else than the one-sided version of Quillen's exact
categories. Some of our statements will be similar with those given for exact categories, as for instance results
proved in detail by B\"uhler~\cite{Buhler}, but several times our proofs will be different, since we shall use a reduced
set of axioms. Let us also note that our concept of strongly one-sided exact category generalizes Rump's notion of
one-sided exact category, and coincides with it provided the category is weakly idempotent complete. We shall usually
refer to (strongly) right exact categories, but the dual results on (strongly) left exact categories hold as well.

The paper is organized as follows. After the preliminaries, we shall define (strongly) one-sided exact categories in
Section 3. In Section 4 we shall show how (strongly) one-sided structures may be transferred between categories. We
shall prove that strongly right exact structures are already exact for a large class of categories, namely for
left quasi-abelian categories, and in particular for abelian categories. On the other hand, we shall show that if $\C$
is a right quasi-abelian category, $\D$ is a coreflective full subcategory of $\C$, and $\E$ is the class of
kernel-cokernel pairs in $\C$ with terms in $\D$, then $\D$ is a right quasi-abelian category and $\E$ gives
rise to a strongly right exact structure on $\D$, which is exact if and only if $\D$ is left quasi-abelian. Explicit
examples of strongly one-sided exact categories which are not exact will be presented. Section 5 will further develop
the theory of right exact categories. We shall study how conflations behave with respect to direct sums and we
give several characterizations of pushouts in right exact categories. Also, we shall show that the Short Five Lemma is
true in right exact categories, generalizing the same result which was known to hold in right quasi-abelian categories
\cite[Lemma~3]{Rump1}. Moreover, we shall prove the $3\times 3$ Lemma in strongly right exact categories. In Section 6
we shall show that weakly idempotent complete strongly right exact categories are right exact in the sense of Rump. Some
further properties on direct sums will be given. In the last section we shall show that the derived category of a right
exact category can be constructed similarly to the derived category of an exact category.

\section{Preliminaries}

In this section we recall some notions and results that will be used throughout the paper. 

The following two lemmas are straightforward. They are dual to \cite[Theorem~5]{RW} and \cite[Example 3, p.~93]{Sten}
respectively, which hold in arbitrary categories. 

\begin{lemma} \label{l:RW} Let $\C$ be a category. Let $i:A\to B$ and $f:A\to A'$ be morphisms in $\C$ such
that $i$ has a cokernel $d:B\to C$, and the pushout of $i$ and $f$ exists. Then the pushout gives rise to the 
commutative diagram 
\[\SelectTips{cm}{}
\xymatrix{
A \ar[d]_f \ar[r]^{i} & B \ar[d]^g \ar[r]^{d} & C \ar@{=}[d] \\ 
A' \ar[r]^{i'} & B' \ar[r]^{d'} & C  
}\] 
in which $d':B'\to C'$ is a cokernel of $i'$.
\end{lemma}

\begin{lemma} \label{l:po} Let $\C$ be a category. Consider a commutative diagram
\[\SelectTips{cm}{}
\xymatrix{
A \ar@{=}[d] \ar[r]^{i} & B \ar[d]^g \ar[r]^{d} & C \ar[d]^h \\ 
A \ar[r]^{i'} & B' \ar[r]^{d'} & C'  
}\] 
in which $d$ is an epimorphism and $d':B'\to C'$ is a cokernel of $i'$. Then the right square is a pushout.
\end{lemma}

\begin{definition} \rm An additive category is called \emph{pre-abelian} if it has kernels and cokernels. It is called
\emph{right (left) quasi-abelian} or \emph{right (left) almost abelian} (\cite{Rump1}, \cite{Rump2}) if it is
pre-abelian and kernels (cokernels) are stable under pushouts (pullbacks), that is, a pushout (pullback) of a kernel
(cokernel) along an arbitrary morphism is a kernel (cokernel). It is called \emph{quasi-abelian} if it is left and right
quasi-abelian.
\end{definition}

\begin{definition} \rm Let $\C$ be a category and let $\D$ be a full subcategory of $\C$. Then $\D$ is
called \emph{reflective} (respectively \emph{coreflective}) if the inclusion functor $i:\D\to \C$ has a left
(respectively right) adjoint. 
\end{definition}

\begin{definition} \rm Let $\A$ be a complete and cocomplete abelian category. Recall (e.g. from
\cite[Chapter~VI, \S1]{Sten}) that a \emph{preradical} on $\mathcal{A}$ is a subfunctor of the identity functor on
$\mathcal{A}$. A preradical $r$ on $\mathcal{A}$ is called a \emph{radical} if $r(A/r(A))=0$ for every object $A$ of
$\mathcal{A}$. 

A full subcategory $\C$ of $\A$  is called a \emph{pretorsion class} if there is a preradical $r$
on $\A$ such that $\C$ consists of the objects $C$ of $\A$ with $r(C)=C$. Dually, $\C$ is called a \emph{pretorsion
free class} if there is a preradical $r$ on $A$ such that $\C$ consists of the objects $C$ of $\A$ with $r(C)=0$. 
\end{definition}

Note that any pretorsion class is closed under coproducts and quotients, whereas any pretorsion free class is closed
under products and subobjects. Pretorsion (pretorsion free) classes are in bijection with idempotent preradicals
(radicals). For more details we refer to \cite[Chapter~VI]{Sten}. Every pretorsion class in $\A$ with
associated preradical $r:\A\to \A$ is a coreflective full subcategory of $\A$, and the right adjoint of the inclusion
functor $i:\C\to {\rm Ab}$ is the functor $r$ viewed as $r:{\rm Ab}\to \C$. Dually, every pretorsion free class in $\A$
is a reflective full subcategory of $\A$ (see \cite[Chapter~X, \S 1]{Sten}). 

To set the terminology we recall the following well-known notions. A morphism $s:A\to B$ in a category $\mathcal{C}$ is
called a \emph{section} if it has a left inverse, that is, there is a morphism $r:B\to A$ in $\mathcal{C}$ such that
$rs=1_A$. A morphism $r:B\to A$ is called a \emph{retraction} if it has a right inverse, that is, there is a morphism
$s:A\to B$ in $\mathcal{C}$ such that $rs=1_A$.

\begin{lemma} {\rm (e.g. \cite{Buhler})} \label{l:secret} The following are equivalent in an additive category:

(i) Every section has a cokernel.

(ii) Every retraction has a kernel.
\end{lemma}

\begin{definition} \rm {\rm (\cite{Buhler}, \cite{Prest})} \label{l:splitid} An additive category
$\mathcal{C}$ is said to \emph{have split idempotents} (or \emph{be idempotent complete}) if for any object $A$ of
$\mathcal{C}$, any idempotent $e=e^2\in \End(A)$ has a kernel. Also, $\mathcal{C}$ is called \emph{weakly
idempotent complete} if the equivalent conditions of Lemma \ref{l:secret} hold.
\end{definition}

\begin{remark}\label{rem:splitid} 
\rm  (1) If $\mathcal{C}$ is an additive category with split idempotents, then $\mathcal{C}$ is weakly idempotent
complete (e.g. see \cite[Remark~6.2]{Buhler}). The converse does not hold in general (see \cite[Exercise~7.11]{Buhler}).

(2) Every additive category has an idempotent-splitting completion (see \cite[p.~7]{Prest}). 
\end{remark}

The next characterization of sections in weakly idempotent complete additive categories easily
follows by \cite[Remark 7.4]{Buhler}. We sketch a proof since we shall use it in Proposition~\ref{p:defl}.

\begin{lemma} \label{l:section} Let $\mathcal{C}$ be a weakly idempotent complete additive category. Then a morphism
$s:A\to B$ in $\mathcal{C}$ is a section if and only if $s$ is isomorphic to $\left [\begin{smallmatrix}
1 \\ 0 \end{smallmatrix}\right ]:A\to A\oplus C$ for some object $C$ of $\mathcal{C}$. 
\end{lemma}

\begin{proof} Assume first that $s:A\to B$ is a section in $\mathcal{C}$. Then there is a morphism $r:B\to A$ such
that $rs=1_A$. By Lemma \ref{l:secret} $s$ has a cokernel, say $p:B\to C$. Since $(1_B-sr)s=0$, there is a unique
morphism $v:C\to B$ such that $vp=1_B-sr$. Then $pvp=p-psr=p$, which implies $pv=1_C$. Also, $rvp=r-rsr=0$, whence
$rv=0$, because $p$ is an epimorphism. Now it is easy to see that $\left [\begin{smallmatrix} r \\ p
\end{smallmatrix}\right ]:B\to A\oplus C$ is an isomorphism with inverse $\left [\begin{smallmatrix} s & v
\end{smallmatrix}\right ]:A\oplus C\to B$. Hence $s:A\to B$ may be identified with $\left [\begin{smallmatrix} 1 \\ 0
\end{smallmatrix}\right ]:A\to A\oplus C$. Similarly for the converse.
\end{proof}

\section{Definition of one-sided exact categories}

We introduce (strongly) one-sided exact categories by considering the one-sided part of Quillen-Keller's axioms
defining exact categories. 

\begin{definition} \rm By a \emph{right exact category} we mean an additive category $\mathcal{C}$ endowed with a
distinguished class of kernels, which are called \emph{inflations} and are denoted by $\rightarrowtail$, satisfying the
following axioms: \vskip2mm

[R0] The identity morphism $1_0:0\to 0$ is an inflation.

[R1] The composition of any two inflations is again an inflation.

[R2] The pushout of any inflation along an arbitrary morphism exists and is again an inflation. \vskip2mm

The pushout of an inflation $i:A\rightarrowtail B$ along $A\to 0$ yields its cokernel $d:B\to C$, which is called
\emph{deflation} and is denoted by $d:B\twoheadrightarrow C$. By \cite[Proposition~3.4, Chapter IV]{Sten}, every
inflation gives rise to a short exact sequence, that is, a kernel-cokernel pair, $A\rightarrowtail B\twoheadrightarrow
C$, which is called \emph{conflation}. An additive category $\mathcal{C}$ is called \emph{left exact} if its
opposite category $\mathcal{C}^{\rm op}$ is right exact and it is called \emph{exact} if it is both left and right
exact.
\end{definition}

\emph{In what follows we shall refer to right exact categories. Obviously, our results are dualizable for left exact
categories.}

\begin{definition} \rm  Let $\C$ be a right exact category. 

We find it convenient to consider a stronger form of axiom [R0], namely: \vskip2mm 

[R0$*$] $0\to A$ is an inflation for every object $A$ of $\C$. \vskip2mm 

We say that $\C$ is  \emph{strongly right exact} if it satisfies the following axiom: \vskip2mm 

[R3] If $i:A\to B$ and $p:B\to C$ are morphisms in $\mathcal{C}$ such that $i$ has a cokernel and $pi$ is an inflation,
then $i$ is an inflation. 
\end{definition}

\begin{remark}\label{rem:exact} \rm  (1) Axiom [R3] is the right part of Quillen's {\it ``obscure axiom''}. 
We point out that by  Keller \cite[Appendix~A]{Keller}, a left and right exact category in our sense coincides
with that of an exact category in the sense of Quillen \cite{Q}, that is, it satisfies axiom [R3] and its dual [R3$^{\rm
op}$]. Moreover, again by \cite[Appendix~A]{Keller}, an additive category is exact if and only if it satisfies [R0],
[R1], [R2] and [R2$^{\rm op}$].

(2) Rump defined a right exact category as an additive category $\C$ endowed with a distinguished class of kernels
satisfying the above axioms [R0]-[R2] and axiom [R3] without asking that $i$ has a cokernel \cite[Definition~4]{Rump2}. 
Proposition \ref{p:obscure} below shows that a category is strongly right exact in our sense if and only if it is right
exact in the sense of Rump, provided the category is weakly idempotent complete. 

(3) A left exact category $\C$ in our sense may be seen as an instance of a Grothendieck pretopology
\cite[Expos\'e II, Definition 1.3]{SGA4}. In general, this pretopology is given by assigning to each object $U$ of
a category $\C$ a family of arrows $\{U_i\to U\}$, called \emph{coverings} of $U$, satisfying certain axioms.
Considering the deflations ending in an object $U$ of a left exact category as the coverings of $U$, the axioms of a
left exact category are nothing else than those of a Grothendieck pretopology.

(4) Rosenberg considered the context of not necessarily additive categories \cite[1.1]{Rosenberg}. The axioms of his
right exact categories mean that a distinguished class of strict epimorphisms yields the coverings of a Grothendieck
pretopology (note the left-right difference between Rosenberg's terminology and ours). 

(5) By the Gabriel-Quillen theorem (see \cite{TT}), every (small) exact category embeds into an abelian category.
Similarly, there is a Gabriel-Quillen type embedding theorem showing that for every (small) right exact category $\C$
there exists an exact category $\C'$ and a fully faithful exact functor (in the sense that it preserves conflations)
from $\C$ to $\C'$ which is universal \cite[Proposition~2.6.1]{Rosenberg}. We shall not make use of this embedding for
transferring properties from exact categories to one-sided exact categories, preferring to give a clearer insight of
one-sided categories through direct proofs from the axioms.
\end{remark}

We have some immediate consequences of the axioms.

\begin{lemma} \label{l:basic} Let $\mathcal{C}$ be a right exact category. Then:

(i) For each object $A$ of $\mathcal{C}$, the identity morphism $1_A$ is an inflation.

(ii) Every isomorphism is an inflation.

(iii) If $\mathcal{C}$ is strongly right exact, then $\C$ satisfies axiom {\rm [R0$*$]}.
\end{lemma}

\begin{proof} (i) Dual to \cite[Appendix A, Step 4]{Keller}.

(ii) See \cite[Remark 2.3]{Buhler}. 

(iii) Since $1_A\colon A\to A$ is a cokernel of $0\to A$ and $1_0\colon 0\to 0$ is the composition of $0\to A$ followed
by $A\to 0$, the morphism $0\to A$ is an inflation by axioms [R0] and [R3].
\end{proof}

\begin{example}\label{e:examples} \rm (1) In any additive category there is an exact structure whose conflations are the
split exact sequences.

(2) In any right quasi-abelian category there is a strongly right exact structure whose
conflations are the kernel-cokernel pairs (e.g., see \cite[Proposition 2 and Corollary 1]{Rump1}).

(3) It is easy to see that the class of all isomorphisms in a category gives rise to a right exact structure which in
general does not satisfy axiom [R0$*$], and so it is not strongly right exact by Lemma \ref{l:basic}.
\end{example}

We illustrate the relations between different types of categories, and (one-sided) exact
categories respectively in the following diagrams:
{\small
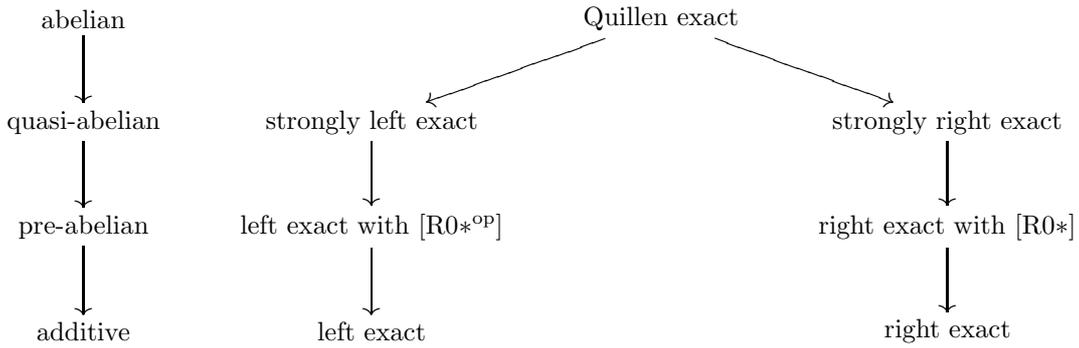
\[\SelectTips{cm}{}
\xymatrix{ 
\textrm{abelian} \ar@{->}[d] & & \textrm{Quillen exact} \ar@{->}[dl] \ar@{->}[dr] & \\
\textrm{quasi-abelian} \ar@{->}[d] & \textrm{strongly left exact} \ar@{->}[d] &  & \textrm{strongly right exact}
\ar@{->}[d] \\ 
\textrm{pre-abelian} \ar@{->}[d] & \textrm{left exact with [R0$*^{\rm op}$]} \ar@{->}[d] & & \textrm{right exact with
[R0$*$]}
\ar@{->}[d] \\ 
\textrm{additive} & \textrm{left exact} &  & \textrm{right exact} 
} \]
}

\section{Constructions of one-sided exact structures}

An extension closed full subcategory $\D$ of an additive exact category $\C$ inherits the exact structure given by all
conflations in $\C$ having terms in $\D$ \cite[Lemma 10.20]{Buhler}. This is also a standard way of constructing
new one-sided exact structures from existent ones. 

\begin{proposition} Let $\C$ be a right exact category and $\D$ an extension closed full subcategory of $\C$.
Then the class of all conflations in $\C$ having terms in $\D$ defines a right exact structure on $\D$.
\end{proposition}

\begin{proof} See \cite[Proposition~2.4.2]{Rosenberg}. 
\end{proof}

We continue with a result which shows how one-sided exact structures may be transferred between categories,
and will be useful for giving examples of one-sided exact structure which are not exact.

\begin{proposition} \label{p:induced} Let $\mathcal{C}$ be a (strongly) right exact category, $\mathcal{D}$ a right
quasi-abelian category, and $L:\mathcal{D}\to \mathcal{C}$ an additive functor which preserves cokernels. Then there is
an induced (strongly) right exact structure on $\mathcal{D}$ defined by the property that a kernel $j$ in $\mathcal{D}$
is an inflation if and only if $L(j)$ is an inflation in $\mathcal{C}$. 
\end{proposition}

\begin{proof} We denote by [R0]-[R2] ([R0]-[R3]) the axioms of the (strongly) right exact category $\mathcal{C}$.
In order to show that the category $\mathcal{D}$ is (strongly) right exact with respect to the defined structure, we
check the corresponding axioms [R0$'$]-[R2$'$] ([R0$'$]-[R3$'$]). 

[R0$'$] Consider the object $0$ of $\mathcal{D}$. Then $1_0$ is the kernel of the morphism $0\to 0$ and
$L(1_0)=1_{L(0)}$ is an inflation in $\mathcal{C}$ by Lemma \ref{l:basic}, and so $1_0$ is an inflation in
$\mathcal{D}$.
 
[R1$'$] Let $j:K\to M$ and $j':M\to N$ be two inflations in $\mathcal{D}$. Then $j'j$ is a kernel by
\cite[Proposition~2]{Rump1}. Also, $L(j):L(K)\to L(M)$ and $L(j'):L(M)\to L(N)$ are inflations in $\mathcal{C}$, and so
$L(j'j)=L(j')L(j)$ is an inflation in $\mathcal{C}$ by [R1]. Hence $j'j$ is an inflation in $\mathcal{D}$. 

[R2$'$] Let $j:K\to M$ and $\alpha:K\to K'$ be morphisms in $\mathcal{D}$ with $j$ an inflation. Let $\left
[\begin{smallmatrix} \beta & j' \end{smallmatrix}\right ]:M\oplus K'\to M'$ be the cokernel of $\left
[\begin{smallmatrix} j \\ -\alpha \end{smallmatrix}\right ]:K\to M\oplus K'$ in $\mathcal{D}$. Then the left diagram
\[\SelectTips{cm}{}
\xymatrix{
K \ar[d]_{\alpha} \ar@{>->}[r]^j & M \ar[d]^{\beta}  \\ 
K' \ar[r]^{j'} & M'  
}  
\hspace{2cm}
\SelectTips{cm}{}
\xymatrix{
L(K) \ar[d]_{L(\alpha)} \ar@{>->}[r]^{L(j)} & L(M) \ar[d]^{L(\beta)}  \\ 
L(K') \ar[r]^{L(j')} & L(M')  
}\]  
is a pushout in $\mathcal{D}$. Since $L$ preserves cokernels, it also preserves pushouts, hence the right diagram is a
pushout in $\mathcal{C}$. Then $j'$ is a kernel because $\mathcal{D}$ is right quasi-abelian. Also, $L(j')$ is an
inflation in $\mathcal{C}$ by [R2]. Hence $j'$ is an inflation in $\mathcal{D}$.

[R3$'$] Let $j:K\to M$ and $p:M\to N$ be morphisms in $\mathcal{D}$ such that $j$ has a cokernel and $pj$ is an
inflation. Then $j$ is a kernel by \cite[Proposition~2]{Rump1}. Also, $L(j):L(K)\to L(M)$ and $L(p):L(M)\to L(N)$ are
morphisms in $\mathcal{C}$ such that $L(j)$ has a cokernel and $L(p)L(j)=L(pj)$ is an inflation in $\mathcal{C}$. Now
[R3] implies that $L(j)$ is an inflation in $\mathcal{C}$, so $j$ is an inflation in $\mathcal{D}$.
\end{proof}

In what follows we shall present some situations when the hypotheses of Proposition \ref{p:induced} and its dual
hold. But first we give the following result, which shows that strongly right exact structures are already exact for a
large class of categories, namely for left quasi-abelian categories, and in particular for abelian categories.

\begin{proposition} \label{p:quasi-ab} If $\C$ is a left quasi-abelian category, then any strongly right exact structure
on $\C$ is exact. 
\end{proposition}

\begin{proof} Assume that $\C$ is a left quasi-abelian strongly right exact category. We have recalled in Remark
\ref{rem:exact} that, in order to prove that $\C$ is exact, it is enough to show [R2$^{\rm op}$]. So let $d:B\to C$ be a
deflation with corresponding inflation $i:A\to B$. Consider the pullback of $d$ along an arbitrary morphism $h:C'\to C$.
By the dual of Lemma \ref{l:RW}, we have a commutative diagram
\[\SelectTips{cm}{}
\xymatrix{
A \ar@{=}[d] \ar[r]^{i'} & B' \ar[d]^g \ar[r]^{d'} & C' \ar[d]^h \\ 
A \ar[r]^i & B \ar[r]^d & C  
}\] 
in which the right square is a pullback and $i':A\to B'$ is a kernel of $d'$. Since $i'$ has a cokernel and $i$ is an
inflation, so is $i'$ by [R3]. By Example \ref{e:examples} (2), the class of all short exact sequences defines a left
exact structure in the left quasi-abelian category $\C$, hence $d'$ is a cokernel. Now by \cite[Proposition~3.4, Chapter
IV]{Sten} we must have $d'=\Coker(i')$, and so $d'$ is a deflation.  
\end{proof}

\begin{proposition} \label{p:sre} Let $\D$ be a right quasi-abelian full subcategory of a right quasi-abelian
category $\C$ such that the inclusion functor $i:\D\to\C$ preserves cokernels. Denote by $\E$ the class of short
exact sequences in $\C$ with terms in $\D$. Then $\E$ gives rise to a strongly right exact structure on $\D$, which is
exact if and only if $\D$ is left quasi-abelian. 
\end{proposition}

\begin{proof} We start with the strongly right exact structure on $\C$ given by all short exact sequences 
(see Example \ref{e:examples} (2)). By Proposition~\ref{p:induced}, $\D$ inherits from $\C$ a strongly right exact
structure with inflations the kernels $f:A\to B$ with $A$ and $B$ objects in $\D$; thus the conflations in $\D$ are
precisely the short exact sequences in the class $\E$. 

If the induced structure is also left exact, then cokernels in $\D$ are stable under pullbacks, that is, $\D$ is
also left quasi-abelian. Conversely, if $\D$ is left quasi-abelian, then by Proposition \ref{p:quasi-ab}, $\E$ induces
an exact structure on $\D$.
\end{proof}

\begin{corollary}\label{c:pretorsion} Let $\C$ be a right quasi-abelian category and let $\D$ be a coreflective
full subcategory of $\C$. Denote by $\E$ the class of short exact sequences in $\C$ with terms in $\D$. Then $\E$ gives
rise to a strongly right exact structure on $\D$, which is exact if and only if $\D$ is left quasi-abelian. 
\end{corollary}

\begin{proof} Since $\D$ is a coreflective full subcategory of $\C$, the inclusion functor $i:\D\to \C$ has a right
adjoint, say $b:\C\to \D$, and we have $ib\cong 1_{\C}$. Then $\D$ is pre-abelian (e.g., by dual results of
\cite[Chapter~X, \S~1]{Sten}, which are also valid in our context). More precisely, if $g:B\to C$ is a morphism in $\C$
with terms in $\D$ and kernel $f:A\to B$ in $\C$, then the restriction $f':b(A)\to B$ of $f$ to $b(A)$ is the kernel of
$f$ in $\D$. Also, the cokernels in $\D$ coincide with the cokernels in $\C$. It is easy to show that kernels
are stable under pushouts in $\D$. Thus $\D$ is right quasi-abelian. Since $i$ preserves cokernels, the conclusion
follows now by Proposition \ref{p:sre}.
\end{proof}

Note that Corollary \ref{c:pretorsion} has a dual that uses reflective subcategories. We now exhibit explicit examples
of coreflective and reflective subcategories which are strongly right or left exact, but not exact.

\begin{example} \label{e:pre-torsion} \rm Let $\Ab$ be the category of abelian groups and let $H$ be the subgroup 
of the group of rational numbers generated by the elements $1/p$, where $p$ varies in the set $P$ of prime numbers.
Consider the idempotent preradical $r:{\rm Ab}\to {\rm Ab}$ given by the trace of $H$ (recall that the trace of $H$ in
an abelian group $G$ is the sum of all images of morphisms from $H$ to $G$). Let $\C$ be the pretorsion class
corresponding to $r$, that is, the class of objects $C$ of Ab such that $r(C)=C$. 

We claim that $\C$ admits a strongly right exact structure which is not left exact. Moreover, we shall show that the
strongly right exact structure may be quite far from being left exact; more precisely, none of the axioms [R1$^{\rm
op}$], [R2$^{\rm op}$], [R3$^{\rm op}$] holds for $\C$. 

First note that since $\C$ is a pretorsion class, it is a coreflective full subcategory of Ab. Then by
Corollary~\ref{c:pretorsion} the monomorphisms $0\to A\to B$ in $\Ab$ with $A$ and $B$ in $\C$ give rise to a strongly
right exact structure in $\C$. A cokernel in $\C$ is a deflation if and only if its kernel in $\Ab$ belongs to $\C$. Let
$P$ be the set of all primes. Note that we have $H/p^n\mathbb{Z}\cong \mathbb{Z}(p^n)$ for every $p\in P$ and every
non-zero natural number $n$. Then it follows that $\C$ contains all torsion groups (in the usual sense) and all
divisible groups. For every $p\in P$, let $\Z(p)$ and $\Z(p^{\infty})$ denote the cyclic group of order $p$ and the
divisible Pr\"ufer $p$-group respectively. It easy to see that $\frac{\prod _{p\in P}\Z(p)}{\bigoplus_{p\in
P}\Z(p)}$ is a divisible torsion free group, and that for every group $G$ with $\bigoplus_{p\in P}\Z(p)\subsetneq G\leq
\prod _{p\in P}\Z(p)$ we have $r(G)=\bigoplus_{p\in P}\Z(p)$. Note that it is enough that one checks this equality for
$G=\prod _{p\in P}\Z(p)$. 


The canonical projection \[\pi\colon \prod _{p\in P}\Z(p^{\infty})\to\frac{\prod _{p\in
P}\Z(p^{\infty})}{\bigoplus_{p\in P}\Z(p)}\] is a deflation, as well as the projection \[\rho\colon \frac{\prod _{p\in
P}\Z(p^{\infty})}{\bigoplus_{p\in P}\Z(p)}\to \frac{\prod _{p\in P}\Z(p^{\infty})}{\prod_{p\in P}\Z(p)},\] but the
kernel of their composition is $\prod_{p\in P}\Z(p)$, which is not an object of $\C$. So [R1$^{\rm op}$] does not hold.

Let $j\colon \Q \to \frac{\prod _{p\in P}\Z(p^{\infty})}{\bigoplus_{p\in P}\Z(p)}$ be the inclusion and consider the
pullback of $\pi$ along $i$ in $\Ab$, that is:
 \[\SelectTips{cm}{}
\xymatrix{
0\ar[r]& \bigoplus_{p\in P}\Z(p)\ar@{=}[d] \ar[r]&Y\ar[d] \ar[r]& \Q\ar[r] \ar[d]^j &0 \\ 
0\ar[r]& \bigoplus_{p\in P}\Z(p)\ar[r]&\prod _{p\in P}\Z(p^{\infty})\ar[r]^{\pi} & \frac{\prod _{p\in
P}\Z(p^{\infty})}{\bigoplus_{p\in P}\Z(p)}\ar[r]&0
}\]  
We have $r(Y)=\bigoplus_{p\in P}\Z(p)$, and so
\[\SelectTips{cm}{}
\xymatrix{
\bigoplus_{p\in P}\Z(p)\ar[d]\ar[r]^0 & \Q\ar[d]^j  \\ 
\prod _{p\in P}\Z(p^{\infty})\ar[r]^{\pi} & \frac{\prod _{p\in P}\Z(p^{\infty})}{\bigoplus_{p\in P}\Z(p)}}\]
is a pullback in $\C$. Note that the lower morphism is a deflation, but the upper morphism is not, so  [R2$^{\rm op}$]
does not hold.

Also, the deflation \[[\pi \ 0]\colon \prod _{p\in P}\Z(p^{\infty})\bigoplus \Q\to \frac{\prod _{p\in
P}\Z(p^{\infty})}{\oplus_{p\in P}\Z(p)}\] is the composition of the morphisms 
\[\SelectTips{cm}{}
\xymatrix{
\prod _{p\in P}\Z(p^{\infty})\oplus \Q \ar[r]^{\left [\begin{smallmatrix} 1&0\\0&0 \end{smallmatrix}\right ]} & \prod
_{p\in P}\Z(p^{\infty})\oplus \Q \ar[r]^{\hspace{5mm}[\pi\  j]} & \frac{\prod _{p\in
P}\Z(p^{\infty})}{\bigoplus_{p\in P}\Z(p)},
}\]  
but $[\pi\ j]$ is not a deflation. This shows that neither the axiom [R3$^{\rm op}$] holds.
\end{example}

\begin{example} \label{e:isbell} 
\rm We adapt and use properties
from \cite{Kelly}. Let $\mathcal{I}$ be the Isbell category, that is, the full subcategory of the
category $\Ab$ of abelian groups consisting of abelian groups having no element of order $p^2$, for a fixed prime number
$p$. $\I$ consists of the abelian groups $G$ whose $p$-primary subgroups are elementary groups.

We claim that $\I$ admits a strongly left exact structure which is not right exact and, similarly to
Example~\ref{e:pre-torsion}, we shall show that none of the axioms [R1], [R2], [R3] holds for $\I$. 

Note that $\I$ is a pretorsion free class. In fact, let $r\colon \Ab\to \Ab$  be the preradical defined by $r(G)=pG_p$,
where $G_p$ is the $p$-primary subgroup of $G$.  It is not difficult to see that $r$ is a radical. As a pretorsion
free class, $\I$ is a reflective full subcategory of Ab. By the dual of Corollary~\ref{c:pretorsion}, the epimorphisms
$B\to C\to 0$ in $\Ab$ with $B$ and $C$ in $\I$ give rise to a strongly left exact structure. Moreover, $\mathcal{I}$ is
complete and cocomplete: limits in $\mathcal{I}$ are formed as in Ab, and colimits are the colimits $G$ in Ab factored
by $pG_p$. 

We shall show that $\I$ is not right exact. A morphism $j$ in $\mathcal{I}$ is an inflation if and only if the cokernel
of $j$ in Ab is a morphism in $\mathcal{I}$. Now $p:\mathbb{Z}\to \mathbb{Z}$ is an inflation in $\mathcal{I}$, but its
composition with itself is not an inflation in $\mathcal{I}$. This shows that axiom [R1] does not hold in $\mathcal{I}$.
Moreover, neither axiom [R2], nor axiom [R3] hold in $\mathcal{I}$. Indeed,  the square
\[\SelectTips{cm}{}
\xymatrix{
\mathbb{Z} \ar[d]_f \ar@{>->}[r]^p & \mathbb{Z} \ar[d]^f  \\ 
\mathbb{Z}(p) \ar[r]^0 & \mathbb{Z}(p) 
}\]  
(with $f\neq 0$ and $\mathbb{Z}(p)$ denoting $\mathbb{Z}/p\mathbb{Z}$) is a pushout and the upper morphism is an
inflation, but the lower morphism is not an inflation. Also, the inflation $\left [\begin{smallmatrix} p\\0
\end{smallmatrix}\right ]:\mathbb{Z}\to \mathbb{Z}\oplus \mathbb{Z}(p)$ is the composition of the morphisms $\left
[\begin{smallmatrix} 1&0\\0&0 \end{smallmatrix}\right ]:\mathbb{Z}\oplus \mathbb{Z}(p)\to \mathbb{Z}\oplus
\mathbb{Z}(p)$ and $\left [\begin{smallmatrix} p \\ \pi \end{smallmatrix}\right ]:\mathbb{Z}\to \mathbb{Z}\oplus
\mathbb{Z}(p)$ (where $\pi:\mathbb{Z}\to \mathbb{Z}(p)$ denotes the canonical projection), but the latter is not an
inflation.
\end{example}

\section{Some homological lemmas}

In this section we study how conflations behave with respect to direct sums and pushouts in right exact
categories, and we employ such properties for proving two important homological results, namely the Short Five Lemma and
the $3\times 3$ Lemma, in a (strongly) right exact category. Our results are inspired by the analogous statements in
\cite{Buhler}, but in general we have to find new proofs, since we use only the one-sided part of the axioms defining an
exact structure. For some results we have to assume axiom [R0$*$] or even [R3].

\begin{lemma} \label{l:ds} Let $\mathcal{C}$ be a right exact category. Then the class of conflations is closed under
isomorphisms and direct sums of short exact sequences.
\end{lemma}

\begin{proof} The first part follows immediately by [R2]. The second part is similar to \cite[Proposition 2.9]{Buhler}
using axiom [R1].
\end{proof}

The next result is similar to a part of \cite[Proposition 3.1]{Buhler}.

\begin{proposition} \label{p:nine} Let $\mathcal{C}$ be a right exact category. Every morphism $(f,g,h)$ between
two conflations $A\stackrel{i}\rightarrowtail B\stackrel{d}\twoheadrightarrow C$ and $A'\stackrel{i'}\rightarrowtail
B'\stackrel{d'}\twoheadrightarrow C'$ factors through some conflation $A'\rightarrowtail D\twoheadrightarrow C$
\[\SelectTips{cm}{}
\xymatrix{
A \ar[d]_f \ar@{>->}[r]^{i} & B \ar[d]^{g'} \ar@{->>}[r]^{d} & C \ar@{=}[d] \\ 
A' \ar@{=}[d] \ar@{>->}[r]^j & D \ar[d]^{g''} \ar@{->>}[r]^p & C \ar[d]^h \\ 
A' \ar@{>->}[r]^{i'} & B' \ar@{->>}[r]^{d'} & C'
}\] 
such that the upper left square and the lower right square of the diagram are pushouts.
\end{proposition}

\begin{proof} As the proof of \cite[Proposition 3.1]{Buhler} using axiom [R2] and Lemmas \ref{l:RW} and \ref{l:po}.
\end{proof}

It is known that the Short Five Lemma (see below) holds in any right quasi-abelian category \cite[Lemma~3]{Rump1},
which has the right exact structure given by all short exact sequences. Also, it holds in any exact
category (see \cite[Corollary 3.2]{Buhler}). We shall generalize this homological lemma to an arbitrary right exact
category. The argument used in the case of exact categories cannot be trasferred in our context. Our proof uses an idea
from \cite[Theorem~6]{RW}.

\begin{lemma}[Short Five Lemma] \label{l:five} Let $\mathcal{C}$ be a right exact category. Consider a morphism
$(f,g,h)$
\[\SelectTips{cm}{}
\xymatrix{
A \ar[d]_f \ar@{>->}[r]^i & B \ar[d]^g \ar@{->>}[r]^d & C \ar[d]^h \\ 
A' \ar@{>->}[r]^{i'} & B' \ar@{->>}[r]^{d'} & C'
}\] 
between two conflations such that $f$ and $h$ are isomorphisms. Then so is $g$.
\end{lemma}

\begin{proof} We claim first that $g$ is an epimorphism. Indeed, if $v:B'\to D$ is a morphism such that $vg=0$, then
$vi'f=vgi=0$, whence $vi'=0$. Then there is a unique morphism $w:C'\to D$ such that $wd'=v$. We have
$whd=wd'g=vg=0$, whence $w=0$, and so $v=0$. Thus $g$ is an epimorphism. 

Now consider the pushout of $i'$ and $if^{-1}$ and use Lemma \ref{l:RW} to obtain the following commutative
diagram:
\[\SelectTips{cm}{}
\xymatrix{
A \ar[d]_f \ar@{>->}[r]^i & B \ar[d]^g \ar@{->>}[r]^d & C \ar[d]^h \\ 
A' \ar@{>->}[r]^{i'} \ar[d]_{if^{-1}} & B' \ar@{->>}[r]^{d'} \ar[d]^{g'} & C' \ar@{=}[d]\\
B \ar@{>->}[r]^{\alpha} & D \ar@{->>}[r]^{\beta} & C'
}\] Then $\alpha$ is an inflation by [R2] and $\beta$ is its cokernel by Lemma \ref{l:RW}, hence the last row is a
conflation. We have $(g'g-\alpha)if^{-1}=g'i'ff^{-1}-g'i'=0$, and so $(g'g-\alpha)i=0$.
Since $d=\Coker(i)$, there is a unique morphism $\gamma:C\to D$ such that $\gamma d=g'g-\alpha$. Then $\beta \gamma
d=\beta g'g-\beta \alpha=d'g=hd$, whence it follows that $\beta \gamma h^{-1}=1$, because $d$ is an epimorphism. Since
$\beta(1_D-\gamma h^{-1}\beta)=0$ and $\alpha=\Ker(\beta)$, there is a unique morphism $\delta:D\to B$ such that
$\alpha\delta=1_D-\gamma h^{-1}\beta$. This implies that $\alpha \delta \alpha=\alpha$ and $\alpha \delta \gamma=0$, and
so $\delta \alpha=1_B$ and $\delta \gamma=0$, because $\alpha$ is a monomorphism. Now we have $\delta g'g=\delta \gamma
d+\delta \alpha=1_B$. On the other hand, the equality $g\delta g'g=g$ implies that $g\delta g'=1_{B'}$, because $g$ is
an epimorphism. This shows that $g$ is an isomorphism with inverse $\delta g'$.
\end{proof}

\begin{proposition} \label{p:part} Let $\mathcal{C}$ be a right exact category. Consider the commutative square
\[\SelectTips{cm}{}
\xymatrix{
A \ar[d]_f \ar@{>->}[r]^i & B \ar[d]^g \\ 
A' \ar@{>->}[r]^{i'} & B'  
}\]
where $i$ and $i'$ are inflations. Then the square is a pushout if and only if it is part of a commutative diagram
\[\SelectTips{cm}{}
\xymatrix{
A \ar[d]_f \ar@{>->}[r]^i & B \ar[d]^g \ar@{->>}[r]^d & C \ar@{=}[d] \\ 
A' \ar@{>->}[r]^{i'} & B' \ar@{->>}[r]^{d'} & C 
}\] 
where the rows are conflations.
\end{proposition}

\begin{proof} Assume first that the square is a pushout. Then the existence of the required commutative diagram follows
by Lemma \ref{l:RW}. Note that $i'$ is an inflation by [R2], and $d'$ is the cokernel of $i'$, so the second row is a
conflation.

Suppose now that the square is part of a commutative diagram as above. By Proposition \ref{p:nine} one obtains a
$3\times 3$ diagram as in the same proposition with $C'=C$. Now by Lemma \ref{l:five} (Short Five Lemma), $g''$ is an
isomorphism, which implies the conclusion.
\end{proof}

\begin{proposition} \label{p:double} Let $\mathcal{C}$ be a right exact category and consider a commutative diagram
\[\SelectTips{cm}{}
\xymatrix{
A \ar@{=}[d] \ar@{>->}[r]^{i} & B \ar[d]^g \ar@{->>}[r]^{d} & C \ar[d]^h \ar@{>->}[r]^{j} & D \ar[d]^f \\ 
A \ar@{>->}[r]^{i'} & B' \ar@{->>}[r]^{d'} & C' \ar@{>->}[r]^{j'} & D'
}\] 
in which the morphisms $i,i',j,j'$ are inflations, $d,d'$ are deflations and the right square is a pushout. Then
the short exact sequence 
\[\xymatrix{B\ar[rr]^-{\left [\begin{smallmatrix} g\\jd \end{smallmatrix}\right ]} &&
B'\oplus D\ar[rr]^-{\left [\begin{smallmatrix} j'd'&-f \end{smallmatrix}\right ]} && D'}\]
is a conflation. 
\end{proposition}

\begin{proof} By Lemma \ref{l:po} the square $BCB'C'$ is a pushout, whence it follows that the rectangle $BDB'D'$
is also a pushout. This is equivalent to the fact that $\left [\begin{smallmatrix} j'd'&-f \end{smallmatrix}\right ]$ is
a cokernel of $\left [\begin{smallmatrix} g\\jd \end{smallmatrix}\right ]$. 

Next let us show that the following commutative square is a pushout:
\[\SelectTips{cm}{}
\xymatrix{
A \ar[d]_{i} \ar@{>->}[r]^{i'} & B' \ar[d]^-{\left [\begin{smallmatrix} 1\\0 \end{smallmatrix}\right ]} \\ 
B \ar[r]^-{\left [\begin{smallmatrix} g\\d \end{smallmatrix}\right ]} & B'\oplus C  
}\] To this end, let $u:B'\to E$ and $v:B\to E$ be morphisms such that $ui'=vi$. Since $(v-ug)i=0$ and
$d=\Coker(i)$, there is a unique morphism $w:C\to E$ such that $wd=v-ug$. Then it is easy to see that $\left
[\begin{smallmatrix} u&w \end{smallmatrix}\right ]:B'\oplus C\to E$ is the unique morphism required for the
pushout property of the previous square. It follows that $\left [\begin{smallmatrix} g\\d \end{smallmatrix}\right ]$ is
an inflation by [R2]. 

Since $j$ is an inflation, so is $\left [\begin{smallmatrix} 1&0\\0&j \end{smallmatrix}\right ]:B'\oplus C\to B'\oplus
C'$ by Lemma \ref{l:ds}. Then $\left [\begin{smallmatrix} g\\jd \end{smallmatrix}\right ]=\left
[\begin{smallmatrix} 1&0\\0&j \end{smallmatrix}\right ] \left [\begin{smallmatrix} g\\d \end{smallmatrix}\right ]$ is an
inflation by [R1]. Now the conclusion follows. 
\end{proof}

In order to complete Proposition \ref{p:part} with other characterizations of pushouts we need now to assume axiom
[R0$*$].

\begin{proposition}\label{p:axiom1/2} Let $\mathcal{C}$ be a right exact category satisfying also {\rm [R0$*$]}.
For every objects $A$ and $B$ of $\mathcal{C}$, the short exact sequence \[\xymatrix{A\ar[r]^-{\left
[\begin{smallmatrix} 1\\0 \end{smallmatrix}\right ]} & A\oplus B\ar[r]^-{\left [\begin{smallmatrix}
0&1\end{smallmatrix}\right ]} & B}\] is a conflation.
\end{proposition}

\begin{proof} By assumption $0\to B$ is an inflation. The required sequence is the direct sum of two
conflations, namely $A\rightarrowtail A\twoheadrightarrow 0$ and $0\rightarrowtail B\twoheadrightarrow B$, hence the
result follows by Lemma \ref{l:ds}.
\end{proof}

\begin{proposition} \label{p:pushout} Let $\mathcal{C}$ be a right exact category satisfying also {\rm [R0$*$]}.
Consider the commutative square
\[\SelectTips{cm}{}
\xymatrix{
A \ar[d]_f \ar@{>->}[r]^i & B \ar[d]^g \\ 
A' \ar@{>->}[r]^{i'} & B'  
}\]
where $i$ and $i'$ are inflations. The following are equivalent:

(i) The square is a pushout.

(ii) The short exact sequence \[\xymatrix{A\ar[r]^-{\left [\begin{smallmatrix} i\\f \end{smallmatrix}\right ]} &
B\oplus A'\ar[r]^-{\left [\begin{smallmatrix} g&-i'\end{smallmatrix}\right ]} & B'}\] is a conflation. 

(iii) The square is both a pushout and a pullback.
\end{proposition}

\begin{proof} Since we are assuming axiom [R0$*$], the proof is the same as the proof of the equivalence of 
the first three conditions in \cite[Proposition 2.12]{Buhler}.
\end{proof}

\begin{corollary} \label{c:infl} Let $\mathcal{C}$ be a right exact category satisfying also {\rm [R0$*$]}, and let
$f:A\to B$ and $f':A\to C$ be morphisms in $\mathcal{C}$ with $f$ an inflation. Then $\left [\begin{smallmatrix} f \\
f' \end{smallmatrix}\right ]:A\to B\oplus C$ is an inflation.
\end{corollary}

\begin{proof} We may consider the pushout of the inflation $f$ and the morphism $f'$. Then $\left [\begin{smallmatrix}
f \\ f' \end{smallmatrix}\right ]$ is an inflation by Proposition \ref{p:pushout}.
\end{proof}

For the next results we need to assume axiom [R3].

\begin{proposition} Let $\mathcal{C}$ be a strongly right exact category. Let $A\stackrel{i}\to B\stackrel{d}\to C$ and
$A'\stackrel{i'}\to B'\stackrel{d'}\to C'$ be composable morphisms such that $A\oplus A'\stackrel{i\oplus
i'}\rightarrowtail B\oplus B'\stackrel{d\oplus d'}\twoheadrightarrow C\oplus C'$ is a conflation. Then the short exact
sequences $A\stackrel{i}\to B\stackrel{d}\to C$ and $A'\stackrel{i'}\to B'\stackrel{d'}\to C'$ are also conflations.
\end{proposition}

\begin{proof} See \cite[Corollary 2.18]{Buhler}.
\end{proof}

\begin{lemma} \label{l:s5lker} Let $\mathcal{C}$ be a strongly right exact category and consider a
commutative diagram
\[\SelectTips{cm}{}
\xymatrix{
A \ar@{=}[d] \ar@{>->}[r]^i & B \ar[d]^g \ar@{->>}[r]^d & C \ar@{>->}[d]^h \\ 
A \ar@{>->}[r]^{i'} & B' \ar@{->>}[r]^{d'} & C'
}\] 
in which the rows are conflations and $h$ is an inflation. Then $g$ is an inflation.
\end{lemma}

\begin{proof} We claim first that $g$ has a cokernel. More precisely, we prove that $\Coker(g)=h'd'$, where
$h'=\Coker(h):C'\to C''$. Let $f:B'\to D$ be a morphism such that $fg=0$. The square $BCB'C'$ is a pushout by Lemma
\ref{l:po}. Hence there is a unique morphism $\gamma:C'\to D$ such that $\gamma d'=f$ and $\gamma h=0$.
Since $h'=\Coker(h)$, there is a unique morphism $\delta:C''\to D$ such that $\delta h'=\gamma$. Then $\delta h'd'=f$,
whence $\Coker(g)=h'd'$.

By Proposition \ref{p:double}, $\left [\begin{smallmatrix} g\\d \end{smallmatrix}\right ]$
is an inflation. Since $h$ is an inflation, so is $\left [\begin{smallmatrix} 1&0\\0&h \end{smallmatrix}\right
]:B'\oplus C\to B'\oplus C'$ by Lemma \ref{l:ds}. Note that $\left [\begin{smallmatrix} 1\\d' \end{smallmatrix}\right
]g=\left [\begin{smallmatrix} 1&0\\0&h \end{smallmatrix}\right ] \left [\begin{smallmatrix} g\\d \end{smallmatrix}\right
]$ is an inflation by [R1]. Since $g$ has a cokernel, axiom [R3] shows that $g$ must be an inflation. 
\end{proof}

Now we are able to prove the $3\times 3$ Lemma in strongly right exact categories. Once again, the proof of the 
same result in exact categories \cite[Corollary~3.6]{Buhler} cannot be transferred to our setting. 

\begin{proposition}[$3\times 3$ Lemma] Let $\mathcal{C}$ be a strongly right exact category and consider the
following commutative diagram:
\[\SelectTips{cm}{}
\xymatrix{
A \ar@{>->}[d]_{f} \ar@{>->}[r]^{i} & B \ar@{>->}[d]^{g} \ar@{->>}[r]^{d} & C \ar@{>->}[d]^{h} \\ 
A' \ar@{->>}[d]_{f'} \ar@{>->}[r]^{i'} & B' \ar@{->>}[d]^{g'} \ar@{->>}[r]^{d'} & C' \ar@{->>}[d]^{h'} \\ 
A'' \ar[r]^{i''} & B'' \ar[r]^{d''} & C''
}\] 
in which the columns and the first two rows are conflations. Then the third row is also a conflation. 
\end{proposition}

\begin{proof} By Proposition \ref{p:nine} the morphism $(f,g,h)$ between the conflations 
$A\stackrel{i}\rightarrowtail B\stackrel{d}\twoheadrightarrow C$ and $A'\stackrel{i'}\rightarrowtail
B'\stackrel{d'}\twoheadrightarrow C'$ factors through some conflation $A'\rightarrowtail D\twoheadrightarrow C$
\[\SelectTips{cm}{}
\xymatrix{
A \ar@{>->}[d]_f \ar@{>->}[r]^{i} & B \ar[d]^{u} \ar@{->>}[r]^{d} & C \ar@{=}[d] \\ 
A' \ar@{=}[d] \ar@{>->}[r]^j & D \ar[d]^{v} \ar@{->>}[r]^p & C \ar@{>->}[d]^h \\ 
A' \ar@{>->}[r]^{i'} & B' \ar@{->>}[r]^{d'} & C'
}\] 
such that the upper left square and the lower right square of the diagram are pushouts. Then $u$ is an inflation,
and by Lemma \ref{l:s5lker}, $v$ is also an inflation. Since $f'f=0=0i$, the pushout property of the square $ABA'D$
yields the existence of a unique morphism $u':D\to A''$ such that $u'j=f'$ and $u'u=0$. By Lemma \ref{l:RW},
$u'=\Coker(u)$. 

Denote $v'=\Coker(v)$. We claim that the following diagram is commutative:  
\[\SelectTips{cm}{}
\xymatrix{
B \ar@{=}[d] \ar@{>->}[r]^{u} & D \ar@{>->}[d]_{v} \ar@{->>}[r]^{u'} & A'' \ar[d]^{i''} \\ 
B \ar@{>->}[r]^{g} & B' \ar@{->>}[d]_{v'} \ar@{->>}[r]^{g'} & B'' \ar[d]^{d''} \\
 & C'' \ar@{=}[r] & C'' 
}\] We already have $vu=g$. The square $DCB'C'$ is a pushout by Lemma \ref{l:po}, whence we have $v'=h'd'=d''g'$ by
Lemma \ref{l:RW}. Furthermore, note that $(i''f')f=0=0i$. Then $i''u',g'v:D\to B''$ are both solutions of the pushout
problem for the square $ABA'D$, because $(i''u')j=i''f'$, $(i''u')u=0$, and also $(g'v)j=i''f'$,
$(g'v)u=0$. Hence we must have $i''u'=g'v$. The square $DA''B'B''$ is a pushout by Lemma \ref{l:po}.
Then $i''$ is an inflation by axiom [R2]. Finally, $d''=\Coker(i'')$ by Lemma \ref{l:RW}. Hence the sequence
$A''\stackrel{i''}\to B''\stackrel{d''}\to C''$ is a conflation. 
\end{proof} 

\begin{remark} \rm The Snake Lemma also holds in a one-sided exact category by \cite[Proposition~C2.2]{Rosenberg}.
\end{remark}

\section{Weakly idempotent complete right exact categories}

First we shall use the characterization of sections in weakly idempotent complete categories from Lemma \ref{l:section}
to obtain a generalization to right exact categories of a result known for exact categories (see \cite[Lemma A.2]{MS}). 

\begin{proposition} \label{p:defl} Let $\mathcal{C}$ be a weakly idempotent complete right exact category and let
$f:A\to B$ and $g:B\to C$ be morphisms in $\mathcal{C}$. Then $\left [\begin{smallmatrix} g & gf \end{smallmatrix}\right
]:B\oplus A\to C$ is a deflation if and only if $g$ is a deflation.
\end{proposition}

\begin{proof} Suppose that $\left [\begin{smallmatrix} g & gf \end{smallmatrix}\right ]:B\oplus A\to C$ is a deflation.
Then the kernel $\left [\begin{smallmatrix} u \\ r \end{smallmatrix}\right ]:K\to B\oplus A$ of $\left
[\begin{smallmatrix} g & gf \end{smallmatrix}\right ]:B\oplus A\to C$ is
an inflation. Consider the isomorphism $\left [\begin{smallmatrix} 1 & f \\ 0 & 1 \end{smallmatrix}\right ]:B\oplus A\to
B\oplus A$ and denote $\left [\begin{smallmatrix} u' \\ r' \end{smallmatrix}\right ]=\left [\begin{smallmatrix} 1 & f \\
0 & 1 \end{smallmatrix}\right ]\cdot \left [\begin{smallmatrix} u \\ r \end{smallmatrix}\right ]$. Then $r'=r$ and we
have an isomorphism of short exact sequences
\[\SelectTips{cm}{}
\xymatrix{
K \ar@{=}[d] \ar@{>->}[r]^-{\left [\begin{smallmatrix} u \\ r \end{smallmatrix}\right ]} & B\oplus A \ar[d]_-{\left
[\begin{smallmatrix} 1 & f \\ 0 & 1
\end{smallmatrix}\right ]} \ar@{->>}[r]^-{\left [\begin{smallmatrix} g & gf \end{smallmatrix}\right ]} & C \ar@{=}[d]
\\ K \ar@{>->}[r]_-{\left [\begin{smallmatrix} u' \\ r \end{smallmatrix}\right ]} & B\oplus A \ar@{->>}[r]_-{\left
[\begin{smallmatrix} g & 0 \end{smallmatrix}\right ]} & C 
}\] 
which implies that the lower sequence is a conflation by Lemma \ref{l:ds}. The lower deflation is determined as being
$\left [\begin{smallmatrix} g & 0 \end{smallmatrix}\right ]$ by the commutativity of the right square. The composition
of the morphisms $\left [\begin{smallmatrix} g & 0 \end{smallmatrix}\right ]:B\oplus A\to C$ and $\left
[\begin{smallmatrix} 0 \\ 1 \end{smallmatrix}\right ]:A\to B\oplus A$ is zero, hence there is a unique morphism $s:A\to
K$ such that $\left [\begin{smallmatrix} u' \\ r \end{smallmatrix}\right ]s=\left [\begin{smallmatrix} 0 \\ 1
\end{smallmatrix}\right ]$. Since $\C$ is weakly idempotent complete, the section $s$ has a cokernel, say $p:K\to D$.
By the proof of Lemma \ref{l:section}, there is a unique morphism $v:D\to K$ such that $vp=1_K-sr$, and $\left
[\begin{smallmatrix} p \\ r \end{smallmatrix}\right ]:K\to D\oplus A$ is an isomorphism. We have
$u'vp=u'$, because $u'-u'vp=u'sr=0$. Thus we obtain the following commutative diagram: 
\[\SelectTips{cm}{}
\xymatrix{
K \ar[d]_-{\left [\begin{smallmatrix} p \\ r \end{smallmatrix}\right ]} \ar@{>->}[r]^-{\left
[\begin{smallmatrix} u' \\ r \end{smallmatrix}\right ]} & B\oplus A \ar@{=}[d] \ar@{->>}[r]^-{\left [\begin{smallmatrix}
g & 0 \end{smallmatrix}\right ]} & C \ar@{=}[d] \\
D\oplus A \ar[d]_-{\left [\begin{smallmatrix} 1 & 0 \end{smallmatrix}\right ]} \ar@{>->}[r]^-{\left [\begin{smallmatrix}
u'v & 0 \\ 0 & 1 \end{smallmatrix}\right ]} & B\oplus A \ar[d]_-{\left [\begin{smallmatrix} 1 & 0
\end{smallmatrix}\right ]} \ar@{->>}[r]^-{\left [\begin{smallmatrix} g & 0 \end{smallmatrix}\right ]} & C \ar@{=}[d] \\
D \ar@{>->}[r]_{u'v} & B \ar@{->>}[r]_{g} & C 
}\] 
The second row is a conflation by Lemma \ref{l:ds}. The lower left square is a pushout, and so $u'v:D\to B$ is an
inflation by [R2]. Moreover, it may be completed by Proposition \ref{p:part} to the above commutative diagram, where the
lower row is a conflation. The lower deflation is determined as being $g$ by the commutativity of the lower right
square.

Conversely, suppose that $g$ is a deflation. Then it has a kernel, say $r:D\to B$, which is an inflation. Then we
have the following commutative diagram: 
\[\SelectTips{cm}{}
\xymatrix{
D \ar[d]_-{\left [\begin{smallmatrix} 1 \\ 0 \end{smallmatrix}\right ]} \ar@{>->}[r]^{r} & B \ar[d]_-{\left
[\begin{smallmatrix} 1 \\ 0 \end{smallmatrix}\right ]} \ar@{->>}[r]^{g} & C \ar@{=}[d] \\ D\oplus A \ar@{>->}[r]_-{\left
[\begin{smallmatrix} r & 0 \\ 0 & 1 \end{smallmatrix}\right ]} & B\oplus A \ar@{->>}[r]_-{\left [\begin{smallmatrix} g &
0 \end{smallmatrix}\right ]} & C 
}\] 
This is because the left square is a pushout, and so $\left [\begin{smallmatrix} r & 0 \\ 0 & 1 \end{smallmatrix}\right
]:D\oplus A\to B\oplus A$ is an inflation by [R2]. Moreover, the square may be completed by Proposition \ref{p:part}
to the above commutative diagram, where the lower row is a conflation. Then we have an isomorphism of short exact
sequences
\[\SelectTips{cm}{}
\xymatrix{
D\oplus A \ar@{=}[d] \ar@{>->}[r]^-{\left [\begin{smallmatrix} r & 0 \\ 0 & 1 \end{smallmatrix}\right ]} & B\oplus A
\ar[d]_-{\left [\begin{smallmatrix} 1 & -f \\ 0 & 1 \end{smallmatrix}\right ]} \ar@{->>}[r]^-{\left [\begin{smallmatrix}
g & 0 \end{smallmatrix}\right ]} & C \ar@{=}[d] \\ D\oplus A \ar@{>->}[r]_-{\left [\begin{smallmatrix} r & -f \\ 0 & 1
\end{smallmatrix}\right ]} & B\oplus A \ar@{->>}[r]_-{\left [\begin{smallmatrix} g & gf \end{smallmatrix}\right ]} & C 
}\] 
in which the lower row is a conflation by Lemma \ref{l:ds}. The lower deflation is determined as being $\left
[\begin{smallmatrix} g & gf \end{smallmatrix}\right ]$ by the commutativity of the right square. Hence $\left
[\begin{smallmatrix} g & gf \end{smallmatrix}\right ]:B\oplus A\to C$ is a deflation.
\end{proof}

We continue with some characterizations of weak idempotent completeness of right exact categories satisfying {\rm
[R0$*$]}.

\begin{lemma} \label{l:wic} Let $\mathcal{C}$ be a  right exact category. Then the following are equivalent:

(i) $\mathcal{C}$ is weakly idempotent complete and satisfies {\rm [R0$*$]}.

(ii) Every section is an inflation.

(iii) Every retraction is a deflation.
\end{lemma}

\begin{proof} Using Lemma~\ref{l:section} the proof is the same as \cite[Corollary 7.5]{Buhler}.
\end{proof}

Now we have the dual of Lemma~\ref{l:basic} (ii).

\begin{corollary} Let $\mathcal{C}$ be a weakly idempotent complete right exact category satisfying also {\rm [R0$*$]}.
Then every isomorphism is a deflation. 
\end{corollary}

We show now that in the case of weakly idempotent complete categories our notion of
strongly one-sided exact category reduces to the notion of one-sided exact category considered by Rump
(\cite[Definition~4]{Rump2}). This amounts to the following proposition. 

\begin{proposition} \label{p:obscure} Let $\mathcal{C}$ be a strongly right exact category. Then the following are
equivalent:

(i) $\mathcal{C}$ is weakly idempotent complete.

(ii) If $i:A\to B$ and $p:B\to C$ are morphisms in $\mathcal{C}$ such that $pi$ is an inflation, then $i$ is an
inflation.  
\end{proposition}

\begin{proof} 
$(i)\Rightarrow(ii)$ By [R2], Lemma~\ref{l:secret} and [R3] the proof is dual to that of \cite[Proposition 7.6]{Buhler}.

$(ii)\Rightarrow(i)$ Let $s:A\to B$ be a section in $\mathcal{C}$. Then there is $r:B\to A$ such that $rs=1_A$. Since
$1_A$ is an inflation by Lemma \ref{l:basic}, so is $s$ by hypothesis. Now $\mathcal{C}$ is weakly idempotent complete
by Lemma \ref{l:wic}. 
\end{proof}

As a consequence, we have a partial converse of Corollary~\ref{c:infl}. It provides a generalization of the dual of
\cite[Lemma 4.2]{MS}.

\begin{corollary} Let $\mathcal{C}$ be a weakly idempotent complete strongly right exact category and
let $f:A\to B$ and $g:B\to C$ be morphisms in $\mathcal{C}$. Then $\left [\begin{smallmatrix} f \\ gf
\end{smallmatrix}\right ]:A\to B\oplus C$ is an inflation if and only if $f$ is an inflation.
\end{corollary}

\begin{proof} Note that $\left [\begin{smallmatrix} f \\ gf \end{smallmatrix}\right ]=\left [\begin{smallmatrix} 1 \\
g \end{smallmatrix}\right ]f$, and use Proposition \ref{p:obscure} and Corollary~\ref{c:infl}.
\end{proof}

Let us point out that one cannot characterize weak idempotent completeness of strongly right
exact categories by the dual of condition (ii) in Proposition \ref{p:obscure}. 

\begin{example} \rm Let $\C$ be the pretorsion class defined in Example~\ref{e:pre-torsion}. Then $\C$ has a
strongly right, but not left, exact structure. Moreover, $\C$ is a right quasi-abelian category, and
consequently weakly idempotent complete. On the other hand, we have seen that axiom [R3$^{\rm op}$] does not hold. 

The Isbell category from Example \ref{e:isbell} gives an example for the dual case of a strongly left exact category. 
\end{example}

Under certain conditions, the class of conflations in a right exact category is closed under arbitrary direct sums (if
they do exist) of short exact sequences. In order to see that, we shall need the following notions.

\begin{definition} \rm An object $I$ of a right exact category $\mathcal{C}$ is called \emph{injective} if for every
inflation $A\rightarrowtail B$, every morphism $A\to I$ extends to a morphism $B\to I$. We say that $\C$ \emph{has
enough injectives} if for every object $A$ of $\mathcal{C}$ there is an inflation $i:A\rightarrowtail I$ with $I$ an
injective object. 
\end{definition}

\begin{example} \rm (1) Let $\C$ be a right quasi-abelian category with enough injectives and let $\D$ be a
coreflective full subcategory of $\C$. Let $b:\C\to \D$ be the right adjoint of the inclusion $i:\D\to \C$. For any
object $D$ of $\D$ one has an inflation $D\to I$ in $\A$ for some injective object $I$ of $\C$. It is easy to see that
$b(I)$ is injective in $\D$, and the induced morphism $D\to b(I)$ is an inflation in $\D$. Hence $\D$ has enough
injectives. For instance, this applies to a pretorsion class in a Grothendieck category. Note that $\D$ might not have
an exact structure (see Example \ref{e:pre-torsion}). 

(2) Let $\mathcal{C}$ be a finitely accessible additive category, that is, an additive category with direct limits such
that the class of finitely presented objects is skeletally small, and every object is a direct limit of finitely
presented objects \cite[p.~8]{Prest}. Then $\mathcal{C}$ has split idempotents \cite[p.~22]{Prest}, and so it is weakly
idempotent complete \cite[Remark~6.2]{Buhler}. Via the Yoneda functor, $\C$ is equivalent to the full subcategory $\F$
of flat objects in the category $\A=(\fp(\mathcal{C})^{\op},\Ab)$ of all contravariant additive functors from the full
subcategory $\fp(\mathcal{C})$ of finitely presented objects of $\mathcal{C}$ to the category ${\rm Ab}$ of abelian
groups, and the pure exact sequences in $\C$ are those which become exact in $\A$ through the Yoneda embedding
\cite[Theorem~3.4]{Prest}. Then $\F$ is extension closed in the abelian category $\A$, and so, the class of all pure
exact sequences in $\C$ gives rise to an exact structure on $\C$ by \cite[Lemma 10.20]{Buhler}. The injective objects in
this exact category are the pure-injective objects, and by \cite[Theorem~6]{Herzog}, for every object $A$ of
$\mathcal{C}$, there is a pure monomorphism $A\rightarrowtail I$ for some pure-injective object $I$ of $\mathcal{C}$.
Hence $\C$ has enough injectives. Note that $\C$ might not be pre-abelian (see \cite[Corollary~3.7]{Prest}). 
\end{example}

We give now a generalization of \cite[Proposition A.6]{MS}, having a similar proof.

\begin{proposition} \label{p:enough} Let $\mathcal{C}$ be a weakly idempotent complete strongly right exact category
with enough injectives. Then:

(i) A morphism $f:A\to B$ is an inflation if and only if the map $\Hom_{\mathcal{C}}(f,I):\Hom_{\mathcal{C}}(B,I)\to
\Hom_{\mathcal{C}}(A,I)$ is an epimorphism of abelian groups for every injective object $I$ of $\mathcal{C}$.
 
(ii) If $(A_k\stackrel{i_k}\rightarrowtail B_k\stackrel{d_k}\twoheadrightarrow C_k)_{k\in K}$ is a family of
conflations having a coproduct, then the short exact sequence 
\[\SelectTips{cm}{}
\xymatrix{
\bigoplus_{k\in K}A_k \ar[rr]^{\oplus_{k\in K}i_k} && \bigoplus_{k\in K}B_k \ar[rr]^{\oplus_{k\in K}d_k} &&
\bigoplus_{k\in K}C_k 
}\] 
is a conflation.
\end{proposition}

\begin{proof} (i) The ``only if'' part is clear. Conversely, let $i:A\rightarrowtail I$ be an inflation in $\mathcal{C}$
with $I$ an injective object. By hypothesis, there is a morphism $g:B\to I$ such that $gf=i$. By Proposition
\ref{p:pushout}, $\left [\begin{smallmatrix} f \\ gf \end{smallmatrix}\right ]=\left [\begin{smallmatrix} f \\ i
\end{smallmatrix}\right ]:A\to B\oplus I$ is an inflation. Finally, by Corollary \ref{c:infl}, $f$ is an
inflation.
 
(ii) Clearly, $\bigoplus_{k\in K}d_k:\bigoplus_{k\in K}B_k\to \bigoplus_{k\in K}C_k$ is the cokernel of $\bigoplus_{k\in
K}i_k$. Hence it is enough to show that $\bigoplus_{k\in K}i_k$ is an inflation. But this follows from part (i), since
$\Hom_{\mathcal{C}}(-,I)$ takes coproducts in $\mathcal{C}$ into products of abelian groups, and products of abelian
groups are exact.
\end{proof}

\section{Derived categories}

In this section we show that the derived category of a right exact category can be constructed similarly to the derived
category of an exact category (see Neeman \cite{Neeman}, Keller \cite{K96}, B\"uhler \cite{Buhler}). The only new thing
to be checked here is that the mapping cone of a chain map between acyclic complexes stays acyclic. 

Let us recall some needed terminology. Throughout $\C$ will be an additive category. Denote by $\mathbf{Ch}(\C)$ the
additive category of complexes and chain maps over $\C$, and by $\mathbf{K}(\C)$ the additive homotopy category whose
objects are the objects of $\mathbf{Ch}(\C)$ and whose morphisms are homotopy classes of morphisms in $\mathbf{Ch}(\C)$.
Recall also that the \emph{mapping cone} ${\rm cone}(f)$ of a chain map $f:A\to B$ in $\mathbf{Ch}(\C)$ is the complex
whose $n$-th component is ${\rm cone}(f)^n=A^{n+1}\oplus B^n$ and $n$-th differential is $d^n_f=\left
[\begin{smallmatrix} -d_A^{n+1}&0 \\ f^{n+1}&d_B^n \end{smallmatrix}\right ]$. Note that the mapping cone defines an
endofunctor of $\mathbf{Ch}(\C)$. 
 
The concept of acyclic chain complex over a right exact category $\C$ can be introduced in the usual way \cite{K96}. 

\begin{definition} \rm Let $\C$ be a right exact category. A chain complex $A$ over $\C$ is called \emph{acyclic} if
each differential $d^{n-1}_A$ factors as 
\[\SelectTips{cm}{}
\xymatrix{
A^{n-1} \ar@{->>}[dr]_{p^{n-1}} \ar[rr]^{d^{n-1}} & & A^n \\ 
& Z^n A \ar@{>->}[ur]_{i^{n-1}} &
}\] 
where $i^{n-1}$ is an inflation that is a kernel of $d^n$ and $p^{n-1}$ is a deflation that is a cokernel of
$d^{n-2}$. 
\end{definition}

Having prepared the needed properties in the previous sections, the following lemma may be proved similarly to 
\cite[Lemma~10.3]{Buhler}. We include a proof since we would like to point out precisely the places where we use our
results for right exact categories. 

\begin{lemma} \label{l:cone} Let $\C$ be a right exact category. Then the mapping cone of a chain map $f:A\to B$ between
acyclic complexes over $\C$ is acyclic.
\end{lemma}

\begin{proof}  It is easy to show that the morphisms $g^n$ and $g^{n+1}$ in the following diagram
\[\SelectTips{cm}{}
\xymatrix{
A^{n-1} \ar[ddd]_{f^{n-1}} \ar@{->>}[dr]_{p^{n-1}_A} \ar[rr]^{d^{n-1}_A} & & A^n \ar[ddd]^{f^n}
\ar@{->>}[dr]_{p^{n}_A} \ar[rr]^{d^{n}_A} & & A^{n+1} \ar[ddd]^{f^{n+1}} \\ 
& Z^n A \ar@{-->}[d]^{g^n} \ar@{>->}[ur]_{i^{n}_A} & & Z^{n+1} A \ar@{-->}[d]^{g^{n+1}} \ar@{>->}[ur]_{i^{n+1}_A} & \\
& Z^n B \ar@{>->}[dr]^{i^n_B} & & Z^{n+1} B \ar@{>->}[dr]^{i^{n+1}_B} & \\
B^{n-1} \ar@{->>}[ur]^{p^{n-1}_B} \ar[rr]_{d^{n-1}_B} & & B^n \ar@{->>}[ur]^{p^{n}_B} \ar[rr]_{d^{n}_B} & & B^{n+1}
}\]  
exist and they are the unique morphisms making the diagram commutative. 

Starting with the morphism $(g^n,f^n,g^{n+1})$ between the conflations $Z^nA \stackrel{i^n_A}\rightarrowtail A^n
\stackrel{j^n_A}\twoheadrightarrow Z^{n+1}A$ and $Z^nB \stackrel{i^n_B}\rightarrowtail B^n
\stackrel{j^n_B}\twoheadrightarrow Z^{n+1}B$, one uses Proposition \ref{p:nine} to obtain some object $Z^nC$ and a
$3\times 3$ commutative diagram in which the middle row is an induced conflation $Z^nB\stackrel{\ 
j^n}\rightarrowtail Z^nC\stackrel{q^n}\twoheadrightarrow Z^{n+1}A$. Then we have the following commutative diagram
\[\SelectTips{cm}{}
\xymatrix{
A^{n-1} \ar[dd]_{f_1^{n-1}} \ar@{->>}[dr]_{p^{n-1}_A} \ar[rr]^{d^{n-1}_A} & & A^n \ar[dd]^{f_1^n}
\ar@{->>}[dr]_{p^{n}_A} \ar[rr]^{d^{n}_A} & & A^{n+1} \ar[dd]^{f_1^{n+1}} \\ 
& Z^n A \ar@{}[dr] | {PO} \ar[dd]^{g^n} \ar@{>->}[ur]_{i^{n}_A} & & Z^{n+1} A \ar@{}[dr] | {PO} \ar[dd]^{g^{n+1}}
\ar@{>->}[ur]_{i^{n+1}_A} & \\
Z^{n-1}C \ar@{}[dr] | {PO} \ar[dd]_{f_2^{n-1}} \ar@{->>}[ur]_{q^{n-1}} & & Z^nC \ar@{}[dr] | {PO} \ar[dd]^{f_2^{n}}
\ar@{->>}[ur]_{q^{n}} & & Z^{n+1}C \ar[dd]^{f_2^{n+1}} \\ 
& Z^n B \ar@{>->}[dr]^{i^n_B} \ar@{>->}[ur]_{j^n} & & Z^{n+1} B \ar@{>->}[dr]^{i^{n+1}_B} \ar@{>->}[ur]_{j^{n+1}} & \\
B^{n-1} \ar@{->>}[ur]^{p^{n-1}_B} \ar[rr]_{d^{n-1}_B} & & B^n \ar@{->>}[ur]^{p^{n}_B} \ar[rr]_{d^{n}_B} & & B^{n+1}
}\]  
in which we have each $f^n=f_2^nf_1^n$ and the quadrilaterals marked by PO are pushouts. 

Now by Proposition \ref{p:double} each sequence 
\[\SelectTips{cm}{}
\xymatrix{
Z^{n}C \ar@{>->}[rr]^-{\left [\begin{smallmatrix} -i_A^{n+1} q^{n} \\ f_2^{n} \end{smallmatrix}\right ]} &&
A^{n+1}\oplus B^{n} \ar@{->>}[rr]^-{\left [\begin{smallmatrix} f_1^{n+1} & j^{n+1} p_B^{n} \end{smallmatrix}\right ]}
&& Z^{n+1}C 
}\]  
is a conflation. We also have commutative diagrams
\[\SelectTips{cm}{}
\xymatrix{
A^n\oplus B^{n-1} \ar@{->>}[dr]_{\left [\begin{smallmatrix}  f_1^{n} & j^{n} p_B^{n-1} \end{smallmatrix}\right ]}
\ar[rr]^{\left [\begin{smallmatrix} -d_A^{n} & 0 \\ f^{n} & d_B^{n-1} \end{smallmatrix}\right ]} & & A^{n+1}\oplus B^n
\\ & Z^n C \ar@{>->}[ur]_{\left [\begin{smallmatrix} -i_A^{n+1} q^{n} \\ f_2^{n} \end{smallmatrix}\right ]} &
}\] 
It follows that the mapping cone of $f$ is acyclic.
\end{proof}

Let $\mathbf{Ac}(\C)$ be the full subcategory of the homotopy category $\mathbf{K}(\C)$ consisting of all acyclic
complexes over a right exact category $\C$. Note that $\mathbf{Ac}(\C)$ is a full additive subcategory of
$\mathbf{K}(\C)$, because Lemma \ref{l:ds} implies that the direct sum of two acyclic complexes is acyclic. Moreover,
Lemma \ref{l:cone} yields the following corollary.

\begin{corollary} Let $\C$ be a right exact category. Then $\mathbf{Ac}(\C)$ is a triangulated subcategory of
$\mathbf{K}(\C)$.
\end{corollary}

Analogously to the derived category of an exact category, we may now define the \emph{derived category} of a right exact
category $\C$ as the Verdier quotient $\mathbf{D}(\C)=\mathbf{K}(\C)/\mathbf{Ac}(\C)$ (see Keller \cite[\S 10, \S
11]{K96}).

\end{document}